\pdfoutput=1
\RequirePackage{ifpdf}
\ifpdf 
\documentclass[pdftex]{sigma}
\else
\documentclass{sigma}
\fi

\numberwithin{equation}{section}

\newtheorem{Theorem}{Theorem}[section]
\newtheorem{Corollary}[Theorem]{Corollary}
\newtheorem{Lemma}[Theorem]{Lemma}
 { \theoremstyle{definition}
\newtheorem{Definition}[Theorem]{Definition}
\newtheorem{Remark}[Theorem]{Remark} }

\def\ta{\theta}

\newcommand{\ti}{\mathrm i}
\newcommand\bn{\mathbf n}
\newcommand\bk{\mathbf k}
\newcommand\bl{\mathbf l}
\newcommand\bx{\mathbf x}

\begin{document}

\allowdisplaybreaks

\newcommand{\arXivNumber}{2005.02203}

\renewcommand{\thefootnote}{}

\renewcommand{\PaperNumber}{088}

\FirstPageHeading

\ShortArticleName{Multidimensional Matrix Inversions and Elliptic Hypergeometric Series on Root Systems}

\ArticleName{Multidimensional Matrix Inversions\\ and Elliptic Hypergeometric Series on Root Systems\footnote{This paper is a~contribution to the Special Issue on Elliptic Integrable Systems, Special Functions and Quantum Field Theory. The full collection is available at \href{https://www.emis.de/journals/SIGMA/elliptic-integrable-systems.html}{https://www.emis.de/journals/SIGMA/elliptic-integrable-systems.html}}}

\Author{Hjalmar ROSENGREN~$^\dag$ and Michael J.~SCHLOSSER~$^\ddag$}

\AuthorNameForHeading{H.~Rosengren and M.J.~Schlosser}

\Address{$^\dag$~Department of Mathematics, Chalmers University of Technology\\
\hphantom{$^\dag$}~and the University of Gothenburg, SE-412 96 G\"oteborg, Sweden}
\EmailD{\href{mailto:hjalmar@chalmers.se}{hjalmar@chalmers.se}}
\URLaddressD{\url{http://www.math.chalmers.se/~hjalmar/}}

\Address{$^\ddag$~Fakult\"at f\"ur Mathematik der Universit\"at Wien,\\
\hphantom{$^\ddag$}~Oskar Morgenstern-Platz 1, A-1090 Wien, Austria}
\EmailD{\href{mailto:michael.schlosser@univie.ac.at}{michael.schlosser@univie.ac.at}}
\URLaddressD{\url{http://www.mat.univie.ac.at/~schlosse/}}

\ArticleDates{Received May 06, 2020, in final form August 28, 2020; Published online September 24, 2020}

\Abstract{Multidimensional matrix inversions provide a powerful tool for studying multiple hypergeometric series. In order to extend this technique to elliptic hypergeometric series, we present three new multidimensional matrix inversions. As applications, we obtain a new~$A_r$ elliptic Jackson summation, as well as several quadratic, cubic and quartic summation formulas.}


\Classification{33D67}

\renewcommand{\thefootnote}{\arabic{footnote}}
\setcounter{footnote}{0}

\section{Introduction}

Explicit matrix inversions provide a powerful tool in the study of special functions.
In particular, they have been used to derive quadratic, cubic and quartic summations for one-variable basic hypergeometric series, see, e.g., \cite{g,gr,gs,rh,rh2}. Moreover,
Andrews \cite{a} found that the Bailey transform, which is useful for deriving Rogers--Ramanujan-type identities,
 is closely related to~a~certain matrix inversion.
A unified generalization of several inversions used in these contexts was obtained by
Krattenthaler \cite{k}. He proved that the lower-triangular matrices $(f_{nk})_{n,k\in\mathbb Z}$ and~$(g_{kl})_{k,l\in\mathbb Z}$ with entries
\begin{gather}\label{ki}
f_{nk}=\frac{\prod\limits_{j=k}^{n-1}(a_j-c_k)}{\prod\limits_{j=k+1}^n(c_j-c_k)},\qquad
g_{kl}=\frac{(a_l-c_l)\prod\limits_{j=l+1}^k(a_j-c_k)}{(a_k-c_k)\prod\limits_{j=l}^{k-1}(c_j-c_k)}
\end{gather}
are mutually inverse.
The goal of the present work is to obtain generalizations of \eqref{ki} that are both multi-dimensional and elliptic, so that we can apply them to multiple elliptic hypergeometric series.
\pagebreak

Although elliptic hypergeometric series first appeared in the 1980s \cite{dj}, they did not take off as a mathematical research area until after the seminal paper of Frenkel and Turaev in 1997 \cite{ft}. In one of the first papers on this subject, Warnaar gave an elliptic extension of Krattenthaler's matrix inverse \cite[Lemma~3.2]{w}
and applied it to obtain elliptic analogues of many identities from \cite{g,gr,gs,rh2}.

\looseness=1
Matrix inversion techniques have also been applied to multiple hypergeometric series.
One then needs multi-dimensional matrices, which in the cases of relevance to us have rows and columns labelled by $\mathbb Z^r$. In the 1990's, Milne and coworkers found several multi-dimensional Bailey transforms and matrix inversions, which were applied to basic hypergeometric series associated with root systems \cite{bm,lm,m,ml,ml2}.
The second author \cite[Theorem~3.1]{s} gave a~unified extension of these inversions, which reduces to \eqref{ki} in the one-variable case. This allowed him to obtain quadratic and cubic summation formulas for multiple hypergeometric series; see
\cite{ks,ls2,ls3,s2,s3,s07,s4,s09} for subsequent work in this direction and \cite{sab} for a~general overview of classical and basic hypergeometric series associated with root systems.

There has been a substantial amount of work on elliptic hypergeometric series associated with root systems, see \cite{rw} and references given there. In particular, Bhatnagar and the second author \cite{bs} obtained several elliptic Bailey transforms associated with root systems, and also gave the corresponding multidimensional matrix inversions explicitly. However, until now these have not been extended to the level of generality of~\eqref{ki}.
In the present paper we address this question.

In Section~\ref{mis} we give two multidimensional extensions of
Warnaar's matrix inversion, as well as one somewhat less general inversion. In Section~\ref{hss} we give applications to summation formulas for elliptic
hypergeometric series on root systems. More precisely, in Section~\ref{jss}
we obtain a new extension of Jackson's summation and two quadratic summations, in all three cases for series associated with the root system $A_r$.
In Section~\ref{dss} we obtain a quadratic, a cubic and a quartic summation for
series traditionally associated with the root system $D_r$. For all these summations except the last one,
we also give companion identities, where the summation is over a simplex rather than a hyper-rectangle.
One of the quadratic $A_r$ summations (Theorem~\ref{aqist}) and the quartic $D_r$ summation (Theorem~\ref{cqust}) are new even in the limit case of basic hypergeometric series.

The multiple hypergeometric series appearing in this paper are of what different authors have called ``type I'', ``Dixon-type'' or ``Gustafson--Milne-type''. For so called ``type II'' or ``Selberg-type'' series, a matrix inversion was found independently by Coskun and Gustafson \cite[equation~(4.16)]{cg} and Rains \cite[Corollary~4.3]{rai0}, see also \cite{c} for the corresponding Bailey transform.
 It~is not clear whether that inversion can be extended to the generality of \eqref{ki}.
A~quadratic summation for Selberg-type series is given in \cite[Corollary~4.3]{rel} and further results are forth\-coming~\cite{lrw}.
There are also related integral identities \cite{vdb, ra,rel,rai2}.
To our knowledge, the present paper contains the first known quadratic summations for multiple elliptic hypergeometric series of type I, and the first known cubic and quartic summations for either type.

There is a close correspondence between elliptic hypergeometric functions given by finite sums and by integrals \cite{rw}.
In particular, Spiridonov and Warnaar~\cite{sw} gave several multiple integral inversions, which can be viewed as continuous analogues of matrix inversions from~\cite{bs}. These integral inversions and their associated Bailey transforms have found a role in quantum field theory, see, e.g., \cite{bs1,bs2,bln,gy,gk,nr,y}. This may serve as a motivation to look for integral analogues of some of our results.

Finally, we mention that it seems
 possible to generalize many of our results to transformation formulas. That will be the subject of future work.\newpage

\section{Preliminaries}

\subsection{Theta functions}

We recall some classical results for theta functions.
We will write
\begin{gather*}
\ta(x;p)=\prod\limits_{j=0}^\infty\big(1-p^jx\big)\big(1-p^{j+1}/x\big),
\end{gather*}
where $p$ is a complex number with $|p|<1$.
We refer to the case
 $p=0$, $\ta(x;0)=1-x$, as the \emph{trigonometric case}.
We will also use the shorthand notation
\begin{gather*}
\theta(x_1,\dots,x_n;p)=\ta(x_1;p)\dotsm\ta(x_n;p),
\end{gather*}
\begin{gather*}
\theta\big(x_1y^{\pm},\dots,x_ny^{\pm};p\big)=\ta(x_1y;p)\ta(x_1/y;p)\dotsm\ta(x_ny;p)
\ta(x_n/y;p).
\end{gather*}
In Section~\ref{gmis}, we will suppress the dependence on $p$ and simply write
$\theta(x)=\theta(x;p)$.

Among the properties of $\theta$ we mention the inversion formula
\begin{gather}
\label{ti}\theta(1/x;p)=-\frac1 x \theta(x;p),
\end{gather}
the quasi-periodicity
\begin{gather}
\label{qpt}\ta(px;p)=-\frac1 x \theta(x;p)
\end{gather}
and Weierstrass' addition formula
\begin{gather}\label{tadd}
\ta(xy,x/y,uv,u/v;p)-\ta(xv,x/v,uy,u/y;p)=\frac uy \ta(yv,y/v,xu,x/u;p).
\end{gather}
We will also need the identity
\begin{gather}\label{cpf}
\sum_{l=1}^k\frac{a_l\prod\limits_{j=1}^{k-2}\theta\big(a_lb_j^\pm;p\big)}
{\prod\limits_{j=1,\,j\neq l}^k\theta\big(a_la_j^\pm;p\big)}=0, \qquad
k\geq 2.
\end{gather}
To our knowledge, it was first obtained by Gustafson
\cite[Lemma~4.14]{gu}, see~\cite{r} for further references and comments.

We will use the following terminology.

\begin{Definition}\label{htfd}
For $k\in\mathbb Z_{\geq 0}$, $0<|p|<1$ and $t\in\mathbb C\setminus{0}$,
a theta function of degree $k$, nome $p$ and norm $t$ is a
holomorphic function on $\mathbb C\setminus\{0\}$ such that
\begin{gather*}
f(pz)=\frac{(-1)^kt}{z^k} f(z),\qquad
z\in\mathbb C\setminus{0}.
\end{gather*}
\end{Definition}

Note that $t$ is not a norm in the sense of normed vector spaces.
It is a classical fact (see, e.g., \cite[Corollary~1.3.5]{rln}) that any function satisfying these conditions can be written as
\begin{gather*}
f(z)=C \theta(z/a_1,\dots,z/a_k;p), \qquad a_1\dotsm a_k=t.
\end{gather*}
In particular, the zero set of $f$ is the union of the geometric progressions $a_jp^{\mathbb Z}$. One consequence is the following classical result.

\begin{Lemma}\label{tvl}
 Let $f$ be a theta function of degree $k$, nome $p$ and norm $t$ and assume that
 \begin{gather*}
 f(b_1)=\dots=f(b_k)=0,
 \end{gather*}
where $b_j$ are non-zero complex numbers such that
\begin{gather*}
b_j/b_k\notin p^{\mathbb Z}, \qquad
j\neq k,\qquad
b_1\dotsm b_k\notin tp^{\mathbb Z}.
\end{gather*}
 Then, $f$ is identically zero.
\end{Lemma}

\subsection{Elliptic hypergeometric series}

A one-variable elliptic hypergeometric series $\sum_k a_k$ is a formal or convergent series such that $a_{k+1}/a_k=f(k)$ for some elliptic function $f$. The standard notation for such series is based on~elliptic shifted factorials
\begin{gather*}
(a;q,p)_k=\prod\limits_{j=0}^{k-1}\theta\big(aq^j;p\big),\qquad
k\in\mathbb Z_{\geq 0}.
\end{gather*}
We will use the condensed notation
\begin{gather*}
(a_1,\dots,a_m;q,p)_k=(a_1;q,p)_k\dotsm (a_m;q,p)_k.
\end{gather*}
Moreover, we will often write $|\mathbf k|=\sum_i k_i$.

Throughout the paper we implicitly assume that all parameters are generic, so that we never divide by zero. In particular, it is convenient to assume that $q^k\notin p^{\mathbb Z}$ for $k\in\mathbb Z_{> 0}$. Otherwise the factor $(q;q,p)_k$, which frequently appears in denominators, vanishes.

We will write
\begin{gather*}
\Delta(\mathbf x;p)=\Delta(x_1,\dots,x_r;p)=\prod\limits_{1\leq i<j\leq r} x_j\theta(x_i/x_j;p).
\end{gather*}
Then, $\Delta(\mathbf x;0)=\prod_{1\leq i<j\leq r} (x_j-x_i)$ can be identified with the Weyl denominator of the root system $A_{r-1}$. Similarly,
$\Delta(\mathbf x;p)$ is essentially the Weyl denominator of the affine root system~$A_{r-1}^{(1)}$~\cite{mac}.

 We will need the $A_r$ elliptic Jackson summation
\begin{gather}
\underset{i=1,\dots,r}{\sum_{0\le k_i\le n_i}}
\frac{\Delta\big(\mathbf x q^{\mathbf k};p\big)}{\Delta(\mathbf x;p)}
\prod\limits_{i=1}^r\bigg(\frac{\ta\big(ax_iq^{|\bk|+k_i};p\big)}{\ta(ax_i;p)}\frac{(ax_i;q,p)_{|\bk|} (dx_i,ex_i;q,p)_{k_i}}
{\big(ax_iq^{n_i+1};q,p\big)_{|\bk|} (ax_iq/b,ax_iq/c;q,p)_{k_i}}\bigg)\nonumber
\\ \qquad\quad{}
\times
\prod\limits_{i,j=1}^r\frac{\big(q^{-n_j}x_i/x_j;q,p\big)_{k_i}}{(qx_i/x_j;q,p)_{k_i}}
\frac{(b,c;q,p)_{|\bk|}}{(aq/d,aq/e;q,p)_{|\bk|}}
q^{|\mathbf k|}\nonumber
\\ \qquad{}\label{r87gl}
=
\frac{(aq/bd,aq/cd;q,p)_{|\bn|}}{(aq/d,aq/bcd;q,p)_{|\bn|}}
\prod\limits_{i=1}^r\frac{(ax_iq,ax_iq/bc;q,p)_{n_i}}{(ax_iq/b,ax_iq/c;q,p)_{n_i}},
\end{gather}
where $a^2q^{|\mathbf n|+1}=bcde$. It is
due to the first author \cite[Corollary~5.3]{r} in general and to Milne~\cite{m} when $p=0$.

If we introduce the parameters $x_{r+1}=a^{-1}$ and $k_{r+1}=-\sum_{i=1}^r k_i$, then some factors from~\eqref{r87gl} can be combined to
\begin{gather*}
\Delta\big(\mathbf x q^{\mathbf k};p\big)\prod\limits_{i=1}^r\ta\big(ax_iq^{|\bk|+k_i};p\big)
=\big(x_{r+1}q^{k_{r+1}}\big)^r\Delta\big(x_1q^{k_1},\dots,x_{r+1}q^{k_{r+1}};p\big).
\end{gather*}
This is the reason why \eqref{r87gl} is associated to $A_r$ rather than $A_{r-1}$.
Rational limit cases of~this type of series play a role in
the representation theory of the corresponding unitary group~$\mathrm{SU}(r+1)$~\cite{hbl}.
In the present paper, we use $A_r$ to label also $r$-dimensional series that are derived from \eqref{r87gl}, even though
 the factor $\prod_{i=1}^r\ta\big(ax_iq^{|\bk|+k_i};p\big)$ is absent, see, e.g.,~\eqref{n87gl}.

We will also need the $C_r$ elliptic Jackson summation
\begin{gather}
\underset{i=1,\dots,r}{\sum_{0\le k_i\le n_i}}
\frac{\Delta\big(\mathbf x q^{\mathbf k};p\big)}{\Delta(\mathbf x;p)}q^{|\mathbf k|}
\prod\limits_{1\le i\le j\le r}\frac{\ta\big(ax_ix_jq^{k_i+k_j};p\big)}{\ta(ax_ix_j;p)}
\prod\limits_{i,j=1}^r\frac{\big(q^{-n_j}x_i/x_j,ax_ix_j;q,p\big)_{k_i}}
{\big(qx_i/x_j,ax_ix_jq^{n_j+1};q,p\big)_{k_i}}\nonumber
\\ \qquad\quad{}
\times
\prod\limits_{i=1}^r\frac{(bx_i,cx_i,dx_i,ex_i;q,p)_{k_i}}
{(ax_iq/b,ax_iq/c,ax_iq/d,ax_iq/e;q,p)_{k_i}}\nonumber
\\ \qquad{}
=\frac{\prod\limits_{i,j=1}^r(ax_ix_jq;q,p)_{n_i}}{\prod\limits_{1\le i<j\le r}\!(ax_ix_jq;q,p)_{n_i+n_j}}\frac{(aq/bc,aq/bd,aq/cd;q,p)_{|\mathbf n|}}
{\prod\limits_{i=1} ^r\!\big(ax_iq/b,ax_iq/c,ax_iq/d,aq^{|\mathbf n|-n_i+1}/bcdx_i;q,p\big)_{n_i}},\!\!
\label{cr87gl}
\end{gather}
where $a^2q^{|\mathbf n|+1}=bcde$. It is
due to the first author \cite[Corollary~5.3]{r} in general and, independently, to
Denis and Gustafson~\cite{dg} and
Milne and Lilly~\cite{ml2} when $p=0$.

\section{Elliptic multidimensional matrix inversions}\label{mis}

\subsection{General form of the matrix inversions}
\label{gmis}

We will consider matrices
 $f=(f_{\bf nk})_{{\bf n}, {\bf k}\in\mathbb Z^r}$
labelled by pairs of multi-indices. All matrices that appear are
lower-triangular in the sense that $f_{\bf nk}=0$ unless $\bf
n\geq\bf k$, that is, $n_i\geq k_i$ for all $i$. Then,
$f$ and $g=(g_{\bf nk})_{{\bf n}, {\bf k}\in\mathbb Z^r}$
are inverse matrices if
\begin{gather}\label{fg}
\sum_{\bf l\leq k\leq n}f_{\bf nk}g_{\bf kl}=\delta_{\bf nl} \end{gather}
or, equivalently,
\begin{gather*}
\sum_{\bf l\leq k\leq n}g_{\bf nk}f_{\bf kl}=\delta_{\bf nl}.
\end{gather*}

Our first main result is the following explicit pair of
 inverse matrices. Since it will be used to~obtain results for multiple hypergeometric series associated with the root system $A_r$, we refer to it as an
 elliptic $A_r$ matrix inversion. Here, and in the remainder of Section~\ref{gmis}, we write
 $\theta(x)=\theta(x;p)$, where $p$ is viewed as fixed.

\begin{Theorem}[an elliptic $A_r$ matrix inversion]\label{amit}
Let
$(a(t))_{t\in\mathbb Z}$ and $(c_j(k))_{k\in\mathbb
Z}$, $1\leq j\leq r$,
 be~arbitrary sequences of scalars. Then the
lower-triangular matrices
\begin{gather*}
f_{\bf nk}=\frac{\prod\limits_{t=|{\bf k}|}^{|{\bf n}|-1}\bigg\{
\theta(a(t)c_1(k_1)\dotsm
c_r(k_r))\prod\limits_{j=1}^r\theta(a(t)/c_j(k_j))\bigg\}}
{\prod\limits_{i=1}^r\prod\limits_{t=k_i+1}^{n_i}\bigg\{\theta(c_i(t)c_1(k_1)\dotsm
c_r(k_r))\prod\limits_{j=1}^r\theta(c_i(t)/c_j(k_j))\bigg\}}
\end{gather*}
and
\begin{gather*}
g_{\bf kl}=
\frac{\theta\big(a({|{\bf l}|)}c_1(l_1)\dotsm c_r(l_r)\big)}
{\theta\big(a({|{\bf k}|})c_1(k_1)\dotsm c_r(k_r)\big)}
\prod\limits_{1\leq i<j\leq r}\frac{\theta(c_i(l_i)/c_j(l_j))}
{\theta(c_i(k_i)/c_j(k_j))}
\prod\limits_{j=1}^r\frac{c_j(l_j)^j\theta(a({|{\bf l}|})/c_j(l_j))}
{c_j(k_j)^j\theta(a({|{\bf k}|})/c_j(k_j))}
\\ \phantom{g_{\bf kl}=}{}
\times\frac{\prod\limits_{t=|{\bf l}|+1}^{|{\bf k}|}\bigg\{
\theta(a(t)c_1(k_1)\dotsm c_r(k_r))\prod\limits_{j=1}^r\theta(a(t)/c_j(k_j))\bigg\}}
{\prod\limits_{i=1}^r\prod\limits_{t=l_i}^{k_i-1}\bigg\{\theta(c_i(t)c_1(k_1)\dotsm
c_r(k_r))\prod\limits_{j=1}^r\theta(c_i(t)/c_j(k_j))\bigg\}}
\end{gather*}
are mutually inverse.
\end{Theorem}

As we explain in more detail in Section~\ref{smis}, Theorem~\ref{amit}
contains two matrix inversions due~to Bhatnagar and the second author \cite{bs} as special cases.
The case $p=0$ of~Theorem~\ref{amit} is~\cite[Theorem~3.1]{s}. (The additional
parameter $b$ present in \cite{s} may be
 suppressed by a change of variables.) It contains several previously
known matrix inversions \cite{bm,lm,m,ml} as special cases.
The case $r=1$ is due to Warnaar \cite[Lemma
3.2]{w}, see also \cite{rs}.

We will prove Theorem~\ref{amit} by verifying \eqref{fg}.
We assume
${\bf n}>{\bf l}$, since otherwise \eqref{fg} is trivial. Pulling out the factors independent of ${\bf k}$, we then have to prove that
\begin{gather*}
\sum_{\bf l\leq k\leq n}\frac{\prod\limits_{t=|{\bf l}|+1}^{|{\bf n}|-1}\bigg\{
\theta(a(t)c_1(k_1)\dotsm
c_r(k_r))\prod\limits_{j=1}^r\theta(a(t)/c_j(k_j))\bigg\}}
{\prod\limits_{i=1}^r\prod\limits_{t=l_i,\,t\neq k_i}^{n_i}\bigg\{\theta(c_i(t)c_1(k_1)\dotsm
c_r(k_r))\prod\limits_{j=1}^r\theta(c_i(t)/c_j(k_j))\bigg\}}
\\ \phantom{\sum_{\bf l\leq k\leq n}}{}
\times\prod\limits_{1\leq i<j\leq r}\frac{1}{\theta(c_i(k_i)/c_j(k_j))}
\prod\limits_{j=1}^r\frac{1}{c_j(k_j)^j}=0.
\end{gather*}
After the substitutions $n_i\mapsto n_i+l_i$, $k_i\mapsto k_i+l_i$, $c_i(k)\mapsto c_i(k-l_i)$ (for $i=1,\dots,r$) and $a(k)\mapsto a(k-|\bf l|)$,
this identity is reduced to the case
 ${\bf l}=\bf 0$. Thus, Theorem~\ref{amit} will
 follow from the following Lemma.

\begin{Lemma}\label{afg}
Let $n_1,\dots,n_r$ be non-negative integers, not all equal to zero,
and let $a(1),\dots$, $a({|{\bf n}|-1})$ and
$c_j(k)$, $1\leq j\leq r$, $0\leq k\leq n_j$, be arbitrary scalars.
Then
\begin{gather*}
\sum_{k_1,\dots,k_r=0}^{n_1,\dots,n_r}\frac{\prod\limits_{t=1}^{|{\bf n}|-1}\bigg\{
\theta(a(t)c_1(k_1)\dotsm
c_r(k_r))\prod\limits_{j=1}^r\theta(a(t)/c_j(k_j))\bigg\}}
{\prod\limits_{i=1}^r\prod\limits_{t=0,\,t\neq k_i}^{n_i}\bigg\{\theta(c_i(t)c_1(k_1)\dotsm
c_r(k_r))\prod\limits_{j=1}^r\theta(c_i(t)/c_j(k_j))\bigg\}}
\\ \phantom{\sum_{k_1,\dots,k_r=0}^{n_1,\dots,n_r}}{}
\times\prod\limits_{1\leq i<j\leq r}\frac{1}{\theta(c_i(k_i)/c_j(k_j))}
\prod\limits_{j=1}^r\frac{1}{c_j(k_j)^j}=0.
\end{gather*}
\end{Lemma}

\begin{proof}
The case $|{\bf n}|=1$ is easily verified. We thus assume
$|{\bf n}|>1$, and prove the result by~induction on $|{\bf n}|$.
Let $f$ be the left-hand
side of the identity, considered as a function of~$a(1)$. We~observe that $f$
 is a theta function of
degree $r+1$, nome $p$ and norm $1$. Moreover, if~\mbox{$a(1)=c_l(t)$}, for some $1\leq l\leq r$ and $0\leq
t\leq n_l$, the term with $k_l=t$ vanishes. The sum is then
reduced to a sum of the same form, but with $n_l$ replaced by $n_l-1$,
and with $a(1)$ and~$c_l(t)$ deleted from the parameter sequences. By
the induction hypothesis, that sum vanishes. This shows that
$f$ has $r+|{\bf n}|$ zeroes, which may without loss of generality be considered generic. It~follows from Lemma \ref{tvl} that $f$ is
 identically
zero.
\end{proof}

\begin{Remark} In the proof we have implicitly assumed $p\neq 0$. When $p=0$, the argument involving theta functions can
be replaced by a polynomial argument. Namely,
in that case $f$ is a polynomial of degree $r+1$ with $r+|{\bf
n}|>r+1$ different zeroes, and thus vanishes identically. The resulting proof
is different from the one given in~\cite{s}.
\end{Remark}

Next we give a matrix inversion that we associate with the root system $BC_r$. It contains a~matrix inversion from \cite{bs} as a special case, see Section~\ref{smis}.

\begin{Theorem}[an elliptic $BC_r$ matrix inversion]\label{bcmit}
Let
$(a(t))_{t\in\mathbb Z}$ and $(c_j(k))_{k\in\mathbb
Z}$, $1 \leq j \leq r$,
 be~arbitrary sequences of scalars.
 Then the
lower-triangular matrices
\begin{gather*}
f_{\bf nk}=\frac{\prod\limits_{j=1}^r\prod\limits_{t=|{\bf k}|}^{|{\bf
n}|-1}\theta\big(c_j(k_j)a(t)^\pm\big)}
{\prod\limits_{i,j=1}^r\prod\limits_{t=k_i+1}^{n_i}\theta\big(c_j(k_j)c_i(t)^\pm\big)}
\end{gather*}
and
\begin{gather*}
g_{\bf kl}=\prod\limits_{1\leq i<j\leq r}
\frac{\theta\big(c_j(l_j)c_i(l_i)^\pm\big)}
{\theta\big(c_j(k_j)c_i(k_i)^\pm\big)}
\prod\limits_{j=1}^r\frac{c_j(k_j)^j\theta\big(c_j(l_j)a({|{\bf l}|})^\pm\big)}
{c_j(l_j)^j\theta\big(c_j(k_j)a({|{\bf k}|})^\pm\big)}
\frac{\prod\limits_{j=1}^r\prod\limits_{t=|{\bf l}|+1}^{|{\bf k}|}
\theta\big(c_j(k_j)a(t)^\pm\big)}
{\prod\limits_{i,j=1}^r\prod\limits_{t=l_i}^{k_i-1}\theta\big(c_j(k_j)c_i(t)^\pm\big)}
\end{gather*}
are mutually inverse.
\end{Theorem}

Theorem~\ref{bcmit} can be obtained as a special case of
 \cite[Theorem~3.1]{s}, that is, of the trigonometric case of
Theorem~\ref{amit}. More precisely, let $p=0$ in Theorem~\ref{amit}, replace $a(t)$
by $ba(t)$ and~$c_j(k)$ by $b c_j(k)$ and then let $b\rightarrow 0$.
This gives the matrix inversion
\begin{gather*}
f_{\bf nk}=\frac{\prod\limits_{t=|{\bf k}|}^{|{\bf n}|-1}\prod\limits_{j=1}^r\left(1-a(t)/c_j(k_j)\right)}
{\prod\limits_{i, j=1}^r\prod\limits_{t=k_i+1}^{n_i}(1-c_i(t)/c_j(k_j))},
\\
g_{\bf kl}=\!\prod\limits_{1\leq i<j\leq r}
\frac{(1-c_i(l_i)/c_j(l_j))}{(1-c_i(k_i)/c_j(k_j))}
\prod\limits_{j=1}^r\frac{c_j(l_j)^j(1-a({|{\bf l}|})/c_j(l_j))}
{c_j(k_j)^j(1-a({|{\bf k}|})/c_j(k_j))}
\frac{\prod\limits_{t=|{\bf l}|+1}^{|{\bf k}|}
\prod\limits_{j=1}^r(1-a(t)/c_j(k_j))}
{\prod\limits_{i,j=1}^r\prod\limits_{t=l_i}^{k_i-1}(1-c_i(t)/c_j(k_j))}.
\end{gather*}
Let us now reintroduce the elliptic nome by substituting
\begin{gather*}
c_j(k)\mapsto \frac{\theta\big(c_j(k)u^\pm\big)}{\theta\big(c_j(k)v^\pm\big)},\qquad a(t)\mapsto\frac{\theta\big(a(t)u^\pm\big)}{\theta\big(a(t)v^\pm\big)}.
\end{gather*}
After applying \eqref{tadd} to all factors, the auxiliary variables $u$ and $v$
only appear in factors that can be cancelled from~\eqref{fg}. This eventually leads to
Theorem~\ref{bcmit}.

It is also possible to prove Theorem~\ref{bcmit} directly in a similar way as
 Theorem~\ref{amit}. The~ana\-logue of Lemma~\ref{afg}
is then the identity
\begin{gather*}
\sum_{k_1,\dots,k_r=0}^{n_1,\dots,n_r}\frac{\prod\limits_{j=1}^r
\bigg\{c_j(k_j)^j\prod\limits_{t=1}^{|{\bf n}|-1}\theta\big(c_j(k_j)a(t)^\pm\big)\bigg\}}
{\prod\limits_{i,j=1}^r\prod\limits_{t=0,\,t\neq
k_i}^{n_i}\theta\big(c_j(k_j)c_i(t)^\pm\big)
\prod\limits_{1\leq i<j\leq r}\theta\big(c_j(k_j)c_i(k_i)^\pm\big)}=0.
\end{gather*}
The special case when $n_j\equiv 1$ and $c_j(0)c_j(1)$ is independent of $j$ is equivalent to \cite[Lemma~7.8]{ra}.
Our proof of Theorem~\ref{amit} is a~straight-forward applications of Rains' method for proving that special case.

Finally, we give a matrix inversion that we associate with the
root system $C_r$. The case $p=0$ of Theorem~\ref{cmit} is \cite[Theorem~4.1]{s}.

\begin{Theorem}[an elliptic $C_r$ matrix inversion]\label{cmit}
Let $(c_j(t))_{t\in\mathbb Z}$, $1 \leq j \leq r$, and $b$ be arbitrary scalars. Then the lower-triangular matrices
\begin{gather*}f_{\bf nk}=\prod\limits_{i=1}^r\frac{\prod\limits_{t=k_i}^{n_i-1}
\bigg\{\theta(c_i(t)b/c_1(k_1)\dotsm c_r(k_r))\prod\limits_{j=1}^r\theta(c_i(t)c_j(k_j))\bigg\}}
{\prod\limits_{t=k_i+1}^{n_i} \bigg\{\theta(c_i(t)c_1(k_1)\dotsm c_r(k_r)/b)
\prod\limits_{j=1}^r\theta(c_i(t)/c_j(k_j))\bigg\}}
\end{gather*}
and
\begin{gather*}
g_{\bf kl}=
\prod\limits_{1\leq i<j\leq r}\frac{\theta\big(c_j(l_j)c_i(l_i)^\pm\big)}
{\theta\big(c_j(k_j)c_i(k_i)^\pm\big)}
\prod\limits_{j=1}^r\frac{c_j(l_j)^{r+1-j}\theta\big(c_j(l_j)^2\big)}
{c_j(k_j)^{r+1-j}\theta\big(c_j(k_j)^2\big)}
\\ \phantom{g_{\bf kl}=}{}
\times\prod\limits_{i=1}^r\frac{\prod\limits_{t=l_i+1}^{k_i}
\bigg\{\theta(c_i(t)b/c_1(k_1)\dotsm
 c_r(k_r))\prod\limits_{j=1}^r\theta(c_i(t)c_j(k_j))\bigg\}}
{\prod\limits_{t=l_i}^{k_i-1}
\bigg\{\theta(c_i(t)c_1(k_1)\dotsm c_r(k_r)/b)
\prod\limits_{j=1}^r\theta(c_i(t)/c_j(k_j))\bigg\}}
\end{gather*}
are mutually inverse.
\end{Theorem}

Note that the one-dimensional case of Theorem~\ref{cmit} is less
general than for Theorems \ref{amit} and \ref{bcmit}. Namely, it
corresponds to the special case when the sequences $a(t)$ and $c_1(t)$
are proportional. We have not been able to extend
Theorem~\ref{cmit} to the level of generality of the
other two inversions. Unfortunately, this means that our method of
proof of those results does not apply to Theorem~\ref{cmit}.
Instead we give another proof which
is more tedious in its details. The same method can be used to
give alternative proofs of
 Theorems~\ref{amit} and~\ref{bcmit}.

As before, to prove Theorem~\ref{cmit}
it is enough to verify \eqref{fg}
for ${\bf n}>\bf l=\bf 0$, which we state as a Lemma.

\begin{Lemma}\label{cfg}
Let $n_1,\dots,n_r$ be non-negative integers, not all equal to zero,
and let $b$ and
$c_j(k)$, $1\leq j\leq r$, $0\leq k\leq n_j$, be arbitrary scalars.
Then
\begin{gather*}
\sum_{k_1,\dots,k_r=0}^{n_1,\dots,n_r}
\prod\limits_{1\leq i<j\leq r}\frac{\theta(c_j(k_j)c_i(k_i))}{\theta(c_j(k_j)/c_i(k_i))}
\prod\limits_{j=1}^r\frac{\theta(c_j(k_j)b/c_1(k_1)\dotsm c_r(k_r))}{c_j(k_j)^{r+1-j}}
\\ \phantom{\sum_{k_1,\dots,k_r=0}^{n_1,\dots,n_r}}
\times\prod\limits_{i=1}^r\frac{\prod\limits_{t=1}^{n_i-1}
\bigg\{\theta(c_i(t)b/c_1(k_1)\dotsm
 c_r(k_r))\prod\limits_{j=1}^r\theta(c_i(t)c_j(k_j))\bigg\}}
{\prod\limits_{t=0,\,t\neq k_i}^{n_i}
\bigg\{\theta(c_i(t)c_1(k_1)\dotsm c_r(k_r)/b)
\prod\limits_{j=1}^r\theta(c_i(t)/c_j(k_j))\bigg\}}=0.
\end{gather*}
\end{Lemma}

\begin{proof}
 Pulling out all factors independent of $k_r$, our
sum can be written as
\begin{gather}\label{ss}
\sum_{k_1,\dots,k_{r-1}=0}^{n_1,\dots,n_{r-1}}A_{k_1,\dots,k_{r-1}}
\sum_{k_r=0}^{n_r}B_{k_1,\dots,k_{r}},
\end{gather}
where
\begin{gather*}
A_{k_1,\dots,k_{r-1}}
=\theta(b/c_1(k_1)\dotsm c_{r-1}(k_{r-1}))\prod\limits_{1\leq i<j\leq
 r-1}\frac{\theta(c_j(k_j)c_i(k_i))}{\theta(c_j(k_j)/c_i(k_i))}
 \\ \phantom{A_{k_1,\dots,k_{r-1}}=}{}
\times\prod\limits_{j=1}^{r-1}\frac{\prod\limits_{i=1}^r\prod\limits_{t=1}^{n_i-1}
\theta(c_i(t)c_j(k_j))}
{c_j(k_j)^{r+1-j}\prod\limits_{i=1}^{r-1}\prod\limits_{t=0,\,t\neq
 k_i}^{n_i}\theta(c_i(t)/c_j(k_j))\prod\limits_{t=0}^{n_r}\theta(c_r(t)/c_j(k_j))},
\\
B_{k_1,\dots,k_{r}}=
\frac{1}{c_r(k_r)}\prod\limits_{j=1}^{r-1}
\theta(c_j(k_j)b/c_1(k_1)\dotsm c_r(k_r),c_j(k_j)c_r(k_r))
\\ \phantom{B_{k_1,\dots,k_{r}}=}{}
\times\prod\limits_{i=1}^r\frac{\prod\limits_{t=1}^{n_i-1}\theta(c_i(t)b/c_1(k_1)\dotsm
 c_r(k_r),c_i(t)c_r(k_r))}{\prod\limits_{t=0,\,t\neq
 k_i}^{n_i}\theta(c_i(t)c_1(k_1)\dotsm
 c_r(k_r)/b,c_i(t)/c_r(k_r))}.
 \end{gather*}

We now rewrite the inner sum in \eqref{ss} using
 Gustafson's identity \eqref{cpf}. Replacing $a_k\mapsto
a_k/\sqrt\lambda$, $b_k\mapsto b_k\sqrt\lambda$ and using \eqref{ti},
\eqref{cpf} can be written as
\begin{gather}\label{cpfv}
\sum_{l=1}^k\frac{\prod\limits_{j=1}^{k-2}\theta(b_ja_l,b_j\lambda/a_l)}
{a_l\prod\limits_{j=1, j\neq l}^k\theta(a_ja_l/\lambda,a_j/a_l)}=0.
\end{gather}
Consider \eqref{cpfv} with $k=|{\bf n}|+1$ and parameters
\begin{gather*}
(a_1,\dots,a_{|{\bf n}|+1})=
\Big(c_1(0),\dots,\widehat{c_1(k_1)},\dots,c_1(n_1),\dots, c_{r-1}(0),\dots,\widehat{c_{r-1}(k_{r-1})},
\dots,
\\ \phantom{(a_1,\dots,a_{|{\bf n}|+1})=(}
c_{r-1}(n_{r-1}), c_r(0),\dots,c_r(n_r)\Big),
\end{gather*}
where the hats signify that the parameter is absent,
\begin{gather*}
(b_1,\dots,b_{|{\bf n}|-1})=
(c_1(1),\dots,c_1(n_1-1),\dots, c_r(1),\dots,c_r(n_r-1),
c_1(k_1),\dots,c_{r-1}(k_{r-1}))
\end{gather*}
and $\lambda=b/c_1(k_1)\dotsm c_{r-1}(k_{r-1})$. Then \eqref{cpfv}
splits naturally into $r$ parts, $\sum_{i=1}^rS_i=0$, where $S_1$ is
the sum of the first $n_1$ terms, $S_2$ the following $n_2$ terms and
so on until the final sum $S_r$ which has $n_r+1$ terms. It is
easy to check that $S_r$ equals the inner sum in \eqref{ss}, so that
\begin{gather*}
\sum_{k_r=0}^{n_r}B_{k_1,\dots,k_{r}}=-\sum_{i=1}^{r-1}S_i.
\end{gather*}

We will complete the proof by showing that
\begin{gather*}\sum_{k_1,\dots,k_{r-1}=0}^{n_1,\dots,n_{r-1}}A_{k_1,\dots,k_{r-1}}
S_i=0,\qquad i=1,\dots,r-1.\end{gather*}
 By a symmetry argument, it suffices to do this for $i=1$.
We have
\begin{gather*}S_1=\sum_{l=0,\,l\neq k_1}^{n_1}C_l, \end{gather*}
with
\begin{gather*}
C_l=
\frac{1}{c_1(l)}
\prod\limits_{j=1}^{r-1}\theta(c_j(k_j)b/c_1(l)c_1(k_1)\dotsm
 c_{r-1}(k_{r-1}),
 c_j(k_j)c_1(l))
 \\ \phantom{C_l=}{}
\times\frac{\prod\limits_{i=1}^r\prod\limits_{t=1}^{n_i-1}\theta
(c_i(t)b/c_1(l)c_1(k_1)\dotsm
 c_{r-1}(k_{r-1}),c_i(t)c_1(l))}
{\prod\limits_{i=1}^{r}\prod\limits_{t=0,\,t\notin \Lambda_i}^{n_i}
\theta(c_i(t)c_1(l) c_1(k_1)\dotsm
 c_{r-1}(k_{r-1})/b,c_i(t)/c_1(l))},
\end{gather*}
where
\begin{gather*}
\Lambda_i=
\begin{cases}\{k_1,l\}, & i=1,\\ \{k_i\}, &
 i=2,\dots,r-1,\\ \emptyset, & i=r.
\end{cases}
\end{gather*}

After some elementary manipulations, we obtain
\begin{gather*}
A_{k_1,\dots,k_{r-1}}C_{l}=
\frac{(-1)^r\theta(c_1(k_1)c_1(l),\gamma/c_1(k_1),
\gamma/c_1(l))}{c_1(k_1)^2c_1(l) \theta(c_1(l)/c_1(k_1))}
\prod\limits_{2\leq i<j\leq
 r-1}\frac{\theta(c_j(k_j)c_i(k_i))}{\theta(c_j(k_j)/c_i(k_i))}
 \\ \phantom{A_{k_1,\dots,k_{r-1}}C_{l}=}{}
\times\prod\limits_{j=2}^{r-1}\frac{\theta(c_j(k_j)c_1(k_1),c_j(k_j)c_1(l),c_j(k_j)\gamma/
c_1(k_1)c_1(l))\prod\limits_{i=1}^r\prod\limits_{t=1}^{n_i-1}
\theta(c_i(t)c_j(k_j))}{c_j(k_j)^{r+2-j}
\prod\limits_{t=0}^{n_1}\theta(c_1(t)/c_j(k_j))\prod\limits_{i=2}^{r}\prod\limits_{t=0,\,t\notin
 \Lambda_i}^{n_i}\theta(c_i(t)/c_j(k_j))}
 \\ \phantom{A_{k_1,\dots,k_{r-1}}C_{l}=}{}
\times\prod\limits_{i=1}^r\frac{\prod\limits_{t=1}^{n_i-1}
\theta(c_i(t)c_1(k_1),c_i(t)c_1(l),c_i(t)\gamma/c_1(k_1)c_1(l))}
{\prod\limits_{t=0,\,t\notin \Lambda_i}^{n_i}
\theta(c_i(t)/c_1(k_1),c_i(t)/c_1(l),c_i(t)c_1(k_1)c_1(l)/\gamma)},
\end{gather*}
where $\gamma=b/c_2(k_2)\dotsm c_r(k_r)$.
This is visibly anti-symmetric under interchanging
$k_1\leftrightarrow l$. Thus,
\begin{gather*}
\sum_{\substack{k_1,l=0\\k_1\neq l}}^{n_1}A_{k_1,\dots,k_{r-1}}C_{l}=0,
\end{gather*}
which completes the proof of the theorem.
\end{proof}

\subsection{Specializations of the matrix inversions}\label{smis}

For applications to hypergeometric series, one typically specializes the parameters in the matrix inversions to geometric progressions. For ease of reference, we give some specializations of this kind explicitly.

We first consider the special case of Theorem~\ref{amit} when
$a(t)=aq^t$ and $c_j(k)=x_jq^{mk}$ ($1\leq j\leq r$),
with $m$ a non-zero integer. It will be convenient to let
\begin{gather*}
F_{\bn \bk}=\frac{g_{\bk \mathbf 0} f_{\bn \bk}}{f_{\bn \mathbf 0}},
\qquad G_{\bk \bl}=\frac{f_{\bl \mathbf 0} g_{\bk \bl}}{g_{\bk \mathbf 0}}.
\end{gather*}
This is another pair of lower-triangular inverse matrices, which has been normalized
so that $F_{\bn \mathbf 0}=G_{\bn \mathbf 0}=1$ for $\mathbf n\geq\mathbf 0$.
Let us first assume that $m$ is positive. After some elementary manipulations, the corresponding special case of Theorem~\ref{amit} takes the following form.

\begin{Corollary}\label{amipc}
When $m$ is a positive integer, the
lower-triangular matrices
\begin{gather*}
F_{\bf nk}=\frac{\Delta(\bx q^{m\bk};p)}{\Delta(\bx;p)}\frac{\big(aXq^{|\bn|};q,p\big)_{m|\bk|}}{(aXq;q,p)_{m|\bk|}} q^{m|\bk|}
\prod\limits_{i,j=1}^r\frac{\big(q^{-mn_j}x_i/x_j;q^m,p\big)_{k_i}}{\big(q^mx_i/x_j;q^m,p\big)_{k_i}}
\\ \phantom{F_{\bf nk}=}{}
\times
\prod\limits_{i=1}^r\bigg(\frac{\theta\big(Xx_iq^{m(|\bk|+k_i)};p\big)}{\theta(Xx_i;p)}
\frac{(x_i/a;q,p)_{mk_i}\big(Xx_i;q^m,p\big)_{|\bk|}}{\big(x_iq^{1-|\bn|}/a;q,p\big)_{mk_i}
\big(Xx_iq^{m(n_i+1)};q^m,p\big)_{|\bk|}}\bigg)
\end{gather*}
and
\begin{gather*}
G_{\bf kl}=\frac{\Delta\big(\bx q^{m\bl};p\big)}{\Delta(\bx;p)}\frac{\theta \big(aXq^{(m+1)|\bl|};p\big)}{\theta(aX;p)}\frac{(aX;q,p)_{|\bl|}}{\big(aXq^{m|\bk|+1};q,p\big)_{|\bl|}} q^{m|\bl|}
\prod\limits_{i,j=1}^r\frac{\big(q^{-mk_j}x_i/x_j;q^m,p\big)_{l_i}}{\big(q^mx_i/x_j;q^m,p\big)_{l_i}}
\\ \phantom{G_{\bf kl}=}{}
\times
\prod\limits_{i=1}^r\bigg(\frac{\theta\big(aq^{|\bl|-ml_i}/x_i;p\big)}{\theta(a/x_i;p)}
\frac{(a/x_i;q,p)_{|\bl|}\big(Xx_iq^{m|\bk|};q^m,p\big)_{l_i}}{\big(aq^{1-mk_i}/x_i;q,p\big)_{|\bl|}
\big(Xx_iq^m;q^m,p\big)_{l_i}}\bigg),
\end{gather*}
where $X=x_1\dotsm x_r$, are mutually inverse.
\end{Corollary}

The case when $m$ is negative can be obtained from
Corollary~\ref{amipc} simply using the natural interpretation
$(a;q,p)_{-n}=\big(aq^{-n};q,p\big)_n^{-1}$ for elliptic shifted factorials with negative subscripts. If we replace $m$ by $-m$ and $x_j$ by $x_j^{-1}$ for $j=1,\dots,r$, the resulting identity takes the following form.

\begin{Corollary}\label{aminc}
When $m$ is a positive integer, the
lower-triangular matrices
\begin{gather*}
F_{\bf nk}=\frac{\Delta\big(\bx q^{m\bk};p\big)}{\Delta(\bx;p)}\frac{(X/a;q,p)_{m|\bk|}}{\big(Xq^{1-|\bn|}/a;q,p\big)_{m|\bk|}} q^{m|\bk|}
\prod\limits_{i,j=1}^r\frac{\big(q^{-mn_j}x_i/x_j;q^m,p\big)_{k_i}}{\big(q^mx_i/x_j;q^m,p\big)_{k_i}}
\\ \phantom{F_{\bf nk}=}{}
\times
\prod\limits_{i=1}^r\bigg(\frac{\theta\big(Xx_iq^{m(|\bk|+k_i)};p\big)}{\theta(Xx_i;p)}
\frac{\big(ax_iq^{|\bn|};q,p\big)_{mk_i}\big(Xx_i;q^m,p\big)_{|\bk|}}{(ax_iq;q,p)_{mk_i}
\big(Xx_iq^{m(n_i+1)};q^m,p\big)_{|\bk|}}\bigg)
\end{gather*}
and
\begin{gather*}
G_{\bf kl}=\frac{\Delta\big(\bx q^{m\bl};p\big)}{\Delta(\bx;p)}\frac{\theta \big(Xq^{(m-1)|\bl|}/a;p\big)}{\theta(X/a;p)}\frac{(a/X;q,p)_{|\bl|}}
{\big(aq^{1-m|\bk|}/X;q,p\big)_{|\bl|}} q^{|\bl|}\prod\limits_{i,j=1}^r
\frac{\big(q^{-mk_j}x_i/x_j;q^m,p\big)_{l_i}}{\big(q^mx_i/x_j;q^m,p\big)_{l_i}}
\\ \phantom{G_{\bf kl}=}{}
\times
\prod\limits_{i=1}^r\bigg(\frac{\theta\big(ax_iq^{|\bl|+ml_i};p\big)}{\theta(ax_i;p)}
\frac{(ax_i;q,p)_{|\bl|}\big(Xx_iq^{m|\bk|};q^m,p\big)_{l_i}}{\big(ax_iq^{mk_i+1};q,p\big)_{|\bl|}
\big(Xx_iq^m;q^m,p\big)_{l_i}}\bigg),
\end{gather*}
where $X=x_1\dotsm x_r$, are mutually inverse.
\end{Corollary}

Finally, we let
$a(t)=aq^t$ and $c_j(k)=x_jq^{mk}$ ($1\leq j\leq r$) in Theorem~\ref{bcmit}.
Since this result is~symmetric under simultaneous inversion of all the parameters $c_j(k)$, we may assume that $m$ is positive. This gives the following result.

\begin{Corollary}\label{bcmic}
When $m$ is a positive integer, the
lower-triangular matrices
\begin{gather*}
F_{\bf nk}=\frac{\Delta\big(\bx q^{m\bk};p\big)}{\Delta(\bx;p)} q^{m|\bk|}
\prod\limits_{1\leq i\leq j\leq r}\frac{\theta\big(x_ix_jq^{m(k_i+k_j)};p\big)}{\theta(x_ix_j;p)}
\\ \phantom{F_{\bf nk}=}{}
\times\prod\limits_{i,j=1}^r\frac{\big(q^{-mn_j}x_i/x_j,x_ix_j;q^m,p\big)_{k_i}}
{\big(q^mx_i/x_j,q^{m(n_j+1)}x_ix_j;q^m,p\big)_{k_i}}
\prod\limits_{i=1}^r\frac{\big(ax_iq^{|\bn|},x_i/a;q,p\big)_{mk_i}}{\big(ax_iq,x_iq^{1-|\bn|}/a;q,p\big)_{mk_i}}
\end{gather*}
and
\begin{gather*}
G_{\bf kl}=\frac{\Delta\big(\bx q^{m\bl};p\big)}{\Delta(\bx;p)}q^{m|\bl|}\prod\limits_{1\leq i<j\leq r}\frac{\theta\big(x_ix_jq^{m(l_i+l_j)};p\big)}{\theta(x_ix_j;p)}
\prod\limits_{i,j=1}^r\frac{\big(q^{-mk_j}x_i/x_j,x_ix_jq^{mk_j};q^m,p\big)_{l_i}}{\big(q^mx_i/x_j,x_ix_j q^m ;q^m,p\big)_{l_i}}
\\ \phantom{G_{\bf kl}=}{}
\times
\prod\limits_{i=1}^r\bigg(\frac{\theta\big(ax_iq^{|\bl|+ml_i},aq^{|\bl|-ml_i}/x_i;p\big)}
{\theta(ax_i,a/x_i;p)}\frac{(ax_i,a/x_i;q,p)_{|\bl|}}{\big(aq^{1+mk_i}x_i,aq^{1-mk_i}/x_i;q,p\big)_{|\bl|}}\bigg),
\end{gather*}
 are mutually inverse.
\end{Corollary}

Since
Theorem~\ref{cmit} lacks the parameters corresponding to $a(t)$ in the other two inversions, we~cannot give a specialization at the same level of generality. If we let
 $c_j(k)=x_jq^{k}$ ($1\leq j\leq r$) we~obtain an inversion equivalent to the case $m=1$ of Corollary~\ref{bcmic}, but with $F$ and $G$ interchanged.

One can check that the explicit matrix inversions obtained by Bhatnagar and the second author \cite[Corollaries~4.5, 5.8 and~9.6]{bs} are equivalent to the case $m=1$ of Corollaries~\ref{amipc}, \ref{aminc} and~\ref{bcmic}, respectively.

\section{Applications to elliptic hypergeometric series}\label{hss}

\subsection{Overview}

If $F$ and $G$ are mutually inverse lower-triangular matrices, and there is a known summation formula
\begin{gather}\label{fab}
\sum_{\bf 0\leq \bk \leq \bn}F_{\bf n\bf k}a_{\bk}=b_{\bn},
\end{gather}
then one can immediately deduce the inverse summation
\begin{gather}\label{gba}
\sum_{\bf 0\leq \bk \leq \bn}G_{\bf n\bf k}b_{\bk}=a_{\bn}.
\end{gather}
We will obtain several new results by applying this procedure to the multiple Jackson summations
\eqref{r87gl} and \eqref{cr87gl}.

In the case of the $A_r$ Jackson summation \eqref{r87gl} this can be done in four ways.
Namely, it~can be identified (in general, or in some special case) with \eqref{fab}, where $F_{\bf n\bf k}$ is as in Corollary~\ref{amipc} or as in Corollary~\ref{aminc}, in both cases with either $m=1$ or $m=2$. In the case $m=1$ of
 Corollary~\ref{aminc}, the identities \eqref{fab} and \eqref{gba} are equivalent. In the other three cases \eqref{gba} gives new summations, see Section~\ref{jss}.

The $C_r$ Jackson summation \eqref{cr87gl} can be expressed as \eqref{fab}, where
 $F_{\bf n\bf k}$ is as in Corollary~\ref{bcmic}, with $1\leq m\leq 4$.
 The sums \eqref{gba} are then of a form traditionally associated with the root system $D_r$ (the motivation for this terminology is weak, but we find it convenient and will use it). The~case $m=1$ gives a new proof of the elliptic $D_r$ Jackson
summation due to the first author \cite[Corollary~6.3]{r} (which for $p=0$ is due independently
to Bhatnagar \cite{b2} and
 the second author \cite{s}).
 This is parallel to the proof in \cite{s}, and we do not provide the details.
 The~remaining three cases lead to new summations, see Section~\ref{dss}.

\subsection[New Ar summations]{New $\boldsymbol{A_r}$ summations}\label{jss}

If we make the change of variables $(a,c)\mapsto (X,aXq^{|\bn}|)$ in \eqref{r87gl}, it can be expressed as \eqref{fab}, where $F_{\bn\bk}$ is as in the case $m=1$ of Corollary~\ref{amipc} and
\begin{gather*}
a_{\bk}=\frac{(aXq,b;q,p)_{|\bk|}}{(Xq/d,abd;q,p)_{\bk}}
\prod\limits_{i=1}^r\frac{(dx_i,Xx_iq/abd;q,p)_{k_i}}{(x_i/a,Xx_iq/b;q,p)_{k_i}},
\\
b_{\bk}=\frac{(Xq/bd,ad;q,p)_{|\bk|}}{(Xq/d,abd;q,p)_{|\bk|}}
\prod\limits_{i=1}^r\frac{\big(Xx_iq,abq^{|\bk|-k_i}/x_i;q,p\big)_{k_i}}
{\big(Xx_iq/b,aq^{|\bk|-k_i}/x_i;q,p\big)_{k_i}}.
\end{gather*}
 In the inverse identity \eqref{gba}, we make the change of variables
$x_i\mapsto tx_i$, $i=1,\dots,r$, $a\mapsto bcdt/aq$, $b\mapsto aq/bc$ and $d\mapsto aq/bdt$, where $t$ is chosen so that after these substitutions
$t^{r+1}={a^2q}/{bcdX}$.
After simplification,
 we can eliminate $t$ and arrive at the following result.

\begin{Theorem}[a new $A_r$ Jackson summation]\label{ijst}
We have the summation formula
\begin{gather}
\underset{i=1,\dots,r}{\sum_{0\le k_i\le n_i}}
\frac{\Delta\big(\mathbf x q^{\mathbf k};p\big)}{\Delta(\mathbf x;p)}
\frac{\ta\big(aq^{2|\bk|};p\big)}{\ta(a;p)}
\frac{(a,b,c;q,p)_{|\bk|}}{\big(aq^{|\bn|+1},aq/b,aq/c;q,p\big)_{|\bk|}}
q^{|\bk|}
\prod\limits_{i,j=1}^r\frac{\big(q^{-n_j}x_i/x_j;q,p\big)_{k_i}}{(qx_i/x_j;q,p)_{k_i}}\nonumber
\\ \phantom{\underset{i=1,\dots,r}{\sum_{0\le k_i\le n_i}}=}{}
\times
\prod\limits_{i=1}^r\frac{(bcd/ax_i;q,p)_{|\bk|-k_i} (d/x_i;q,p)_{|\bk|}
\big(a^2x_iq^{|\bn|+1}/bcd;q,p\big)_{k_i}}
{(d/x_i;q,p)_{|\bk|-k_i} (bcdq^{-n_i}/ax_i;q,p)_{|\bk|} (ax_iq/d;q,p)_{k_i}}\nonumber
\\ \phantom{\underset{i=1,\dots,r}{\sum_{0\le k_i\le n_i}}}{}
=\frac{(aq,aq/bc;q,p)_{|\bn|}}{(aq/b,aq/c;q,p)_{|\bn|}}
\prod_{i=1}^r\frac{(ax_iq/bd,ax_iq/cd;q,p)_{n_i}}{(ax_iq/d,ax_iq/bcd;q,p)_{n_i}}.\label{n87gl}
\end{gather}
\end{Theorem}

Theorem~\ref{ijst} has a companion identity, where the summation is supported on a simplex.
\begin{Corollary}\label{ijsc} For parameters subject to the relation
$a^2q^{N+1}=bc_1\dotsm c_{r+1}dx_1\dotsm x_r$,
we have the summation formula
\begin{gather}
\underset{0\le|\bk|\le N}{\sum_{k_1,\dots,k_r\ge 0}}
\frac{\Delta(\mathbf x q^{\mathbf k};p)}{\Delta(\mathbf x;p)}
\frac{\ta\big(aq^{2|\bk|};p\big)}{\ta(a;p)}
\frac{\big(a,b,q^{-N};q,p\big)_{|\bk|}}{\big(aq/b,aq^{N+1};q,p\big)_{|\bk|}} q^{|\bk|}\nonumber
\\ \phantom{\underset{0\le|\bk|\le N}{\sum_{k_1,\dots,k_r\ge 0}}=}{}
\times
\prod\limits_{i=1}^r\frac{(aq/CXx_i;q,p)_{|\bk|-k_i} (d/x_i;q,p)_{|\bk|}
\prod\limits_{j=1}^{r+1}(c_jx_i;q,p)_{k_i}}
{(d/x_i;q,p)_{|\bk|-k_i} (ax_iq/d;q,p)_{k_i}\prod\limits_{j=1}^r(qx_i/x_j;q,p)_{k_i}}
\prod\limits_{i=1}^{r+1}\frac 1{(ac_iq/CX;q,p)_{|\mathbf k|}}\nonumber
\\ \phantom{\underset{0\le|\bk|\le N}{\sum_{k_1,\dots,k_r\ge 0}}}{}
=\frac{(aq;q,p)_{N}}{b^N(aq/b;q,p)_{N}}
\prod\limits_{i=1}^r\frac{(ax_iq/bd;q,p)_{N}}{(ax_iq/d;q,p)_{N}}
\prod\limits_{i=1}^{r+1}\frac{(aq/c_id;q,p)_{N}}{(aq/bc_id;q,p)_{N}},\label{n87cgl}
\end{gather}
where $C=c_1\cdots c_{r+1}$ and $X=x_1\dotsm x_r$.
\end{Corollary}

We sketch the standard argument for deducing
Corollary~\ref{ijsc} from Theorem~\ref{ijst}. One first observes that the case $c_j=q^{-n_j}/x_j$ (for $1\leq j\leq r$) of Corollary~\ref{ijsc} is equivalent to the case $c=q^{-N}$ of Theorem~\ref{ijst}. One then uses the quasi-periodicity \eqref{qpt} to deduce that Corollary~\ref{ijsc} holds whenever $c_jx_j\in p^{\mathbb Z}q^{\mathbb Z_{\leq 0}}$. The general case then follows by analytic continuation in the parameters $c_j$.

The case $p=0$ of Theorem~\ref{ijst} and Corollary~\ref{ijsc} are due to the second author \cite[Theorem~4.1, Corollary~4.2]{s4}. They
generalize multiple ${}_6\phi_5$ summations due to
Bhatnagar \cite{b}.
The general case of Theorem~\ref{ijst} has already been announced and applied in several publications.
In~\cite{r3}, the first author found that it appears in connection with Felder's $\mathrm{SU}(2)$ elliptic quantum group. This led to an explicit system of multivariable biorthogonal functions, which essentially have the summand in \eqref{n87gl} as their weight function. In~\cite[Theorem~5.2]{bs}, Bhatnagar and the second author derived a multiple Bailey transformation\footnote{This refers to a generalization of
Bailey's ${}_{10}W_9$-transformation and should not be confused with the notion of Bailey transform mentioned in the Introduction.} by combining \eqref{r87gl} and \eqref{n87gl}
(the~same result was independently obtained by the first author). Finally, in \cite[Theorem~4.1]{r4} the~first author found another multiple Bailey transformation by combining \eqref{n87gl}
with yet ano\-ther (Gustafson--Rakha-type) multiple Jackson summation.
In fact, we have found a third multiple Bailey transformation
by combining \eqref{n87gl} with itself, but we save that result for future work.

Next, we replace $q$ by $q^2$ in \eqref{r87gl} and then make
the substitutions
\begin{gather*}
(a,b,c)\mapsto \big(X,aXq^{|\bn|},aXq^{|\bn|+1}\big).
\end{gather*}
Using the identity $(a,aq;q^2,p)_k=(a;q,p)_{2k}$,
we find that
 \eqref{fab} holds, with $F_{\bn\bk}$ as in the case $m=2$ of Corollary~\ref{amipc} and
\begin{gather*}
a_{\bk}=\frac{(aXq;q,p)_{2|\bk|}}{\big(Xq^2/d,a^2dXq;q^2,p\big)_{|\bk|}}
\prod\limits_{i=1}^r\frac{\big(dx_i,x_iq/a^2d;q^2,p\big)_{k_i}}{(x_i/a;q,p)_{2k_i}},
\\
b_{\bk}=q^{-e_2(\bk)}\frac{(ad,q/ad;q,p)_{|\bk|}}{\big(Xq^2/d,a^2 d Xq;q^2,p\big)_{|\bk|}}
\prod\limits_{i=1}^r\frac{\big(X x_iq^2,a^2 Xq^{2|\bk|-2k_i+1}/x_i;q^2,p\big)_{k_i}}{\big(aq^{|\bk|-k_i}/x_i,x_iq^{-|\bk|+k_i+1}/a;q,p\big)_{k_i}},
\end{gather*}
where $e_2(\mathbf k)=\sum_{1\leq i<j\leq r}k_ik_j$.
 In the inverse identity~\eqref{gba}, we make the change of variables $x_i\mapsto tx_i$, $i=1,\dots,r$,
$a\mapsto dt/aq$ and $d\mapsto abq/dt$, where $t$ is chosen so that after these substitutions
$t^{r+1}=a^2q/dX$. After simplification, this gives the following result.

\begin{Theorem}\label{aqist}
If $a^2q^{2|\bn|+1}=cd$, we have the quadratic summation formula
\begin{gather*}
\sum_{\substack{0\leq k_i\leq n_i,\\i=1,\dots,r}}
\frac{\Delta\big(\bx q^{2\bk};p\big)}{\Delta(\bx;p)}
\frac{\theta\big(aq^{3|\bk|};p\big)}{\theta(a;p)}
\frac{(a,b,q/b;q,p)_{|\bk|}}{\big(aq^{2|\bn|+1};q,p\big)_{|\bk|}\big(aq^2/b,aqb;q^2,p\big)_{|\bk|}} q^{|\bk|-e_2(\bk)}
\\ \phantom{\sum_{\substack{0\leq k_i\leq n_i,\\i=1,\dots,r}}}{}
\times
\prod\limits_{i=1}^r\frac{\big(cx_i;q^2,p\big)_{k_i}\big(d/x_i;q^2,p\big)_{|\bk|}(d/ax_i;q,p)_{|\bk|-k_i}}
{\big(aq^{2|\bn|-2n_i+1}/cx_i;q,p\big)_{|\bk|}\big(d/x_i;q^2,p\big)_{|\bk|-k_i}\big(ax_iq^{k_i-|\bk|+1}/d;q\big)_{k_i}}
\\ \phantom{\sum_{\substack{0\leq k_i\leq n_i,\\i=1,\dots,r}}}{}
\times\prod\limits_{i,j=1}^r\frac{\big(q^{-2n_j}x_i/x_j;q^2,p\big)_{k_i}}{\big(q^2x_i/x_j;q^2,p\big)_{k_i}}
=\frac{(aq;q,p)_{2|\bn|}}{\big(abq,aq^2/b;q^2,p\big)_{|\bn|}}\prod\limits_{i=1}^r
\frac{\big(abx_iq/d,ax_iq^2/bd;q^2,p\big)_{n_i}}{(ax_iq/d;q,p)_{2n_i}}.
\end{gather*}
\end{Theorem}

By a standard argument (see the discussion of Corollary~\ref{ijsc} above) we deduce
 the following companion identity.

\begin{Corollary}\label{aqisc}
If
$a^2q=b_1\dotsm b_{r+1}cx_1\dotsm x_r$,
we have the quadratic summation formula
\begin{gather*}
\sum_{\substack{k_1,\dots,k_r\geq 0,\\|\bk|\leq N}}
\frac{\Delta\big(\bx q^{2\bk};p\big)}{\Delta(\bx;p)}\frac{\theta\big(aq^{3|\bk|};p\big)}{\theta(a;p)}
\frac{\big(a,q^{-N},q^{N+1};q,p\big)_{|\bk|}}{\big(aq^{N+2},aq^{1-N};q^2,p\big)_{|\bk|}
\prod\limits_{i=1}^{r+1}(ab_iq/BX;q,p)_{|\bk|}} q^{|\bk|-e_2(\bk)}
\\ \phantom{\sum_{\substack{k_1,\dots,k_r\geq 0,\\|\bk|\leq N}}}{}
\times\prod\limits_{i=1}^r\frac{\big(c/x_i;q^2,p\big)_{|\bk|}(c/ax_i;q,p)_{|\bk|-k_i}
\prod\limits_{j=1}^{r+1}\big(b_jx_i;q^2,p\big)_{k_i}}{\big(c/x_i;q^2,p\big)_{|\bk|-k_i}
\big(ax_iq^{k_i-|\bk|+1}/c;q,p\big)_{k_i}\prod\limits_{j=1}^r\big(q^2x_i/x_j;q^2,p\big)_{k_i}}
\\ \phantom{\sum_{\substack{k_1,\dots,k_r\geq 0,\\|\bk|\leq N}}}{}
=\begin{cases}\displaystyle\frac{\big(aq^2;q^2,p\big)_m}{\big(q/a;q^2,p\big)_m}
\prod\limits_{i=1}^r\frac{\big(cq/ax_i;q^2,p\big)_m}{\big(ax_iq^2/c;q^2,p\big)_m}
\prod\limits_{i=1}^{r+1}\frac{\big(aq^{2}/b_ic;q^2,p\big)_m}{\big(b_icq/a;q^2,p\big)_m}, &N=2m,
\\[2.5ex]
\displaystyle\frac{\big(aq;q^2,p\big)_m}{\big(1/a;q^2,p\big)_m}\prod\limits_{i=1}^r
\frac{\big(c/ax_i;q^2,p\big)_m}{\big(ax_iq/c;q^2,p\big)_m}\prod\limits_{i=1}^{r+1}
\frac{\big(aq/b_ic;q^2,p\big)_m}{\big(b_ic/a;q^2,p\big)_m}, &N=2m-1,
\end{cases}
\end{gather*}
where $B=b_1\dotsm b_{r+1}$ and $X=x_1\dotsm x_r$.
\end{Corollary}

Theorem~\ref{aqist} and Corollary~\ref{aqisc} are new even in the case $p=0$. When $r=1$, they reduce to the case $a=b$ of \cite[Corollary~4.4]{w} and to \cite[Corollary~4.10]{w}, respectively.
The case when both $p=0$ and $r=1$ is due to Gessel and Stanton \cite{gs} and, independently, to Gasper and Rahman \cite{g,rh2}, respectively.

Finally, we replace $q$ by $q^2$ in \eqref{r87gl} and then make the substitutions
\begin{gather*}
(a,d,e)\mapsto \big(X,aXq^{|\bn|},aXq^{|\bn|+1}\big).
\end{gather*}
 The resulting identity takes the form
 \eqref{fab}, where $F_{\bn\bk}$ is as in the case $m=2$ of Corollary~\ref{aminc} and
\begin{gather*}
a_{\bk}=\frac{\big(b,c;q^2,p\big)_{|\bk|}}{(X/a;q,p)_{2|\bk|}}
\prod\limits_{i=1}^r\frac{(ax_iq;q,p)_{2k_i}}{\big(Xx_iq^2/b,Xx_iq^2/c;q^2,p\big)_{k_i}},
\\
b_{\bk}=\frac{(ab/X,Xq/ab;q,p)_{|\bk|}}{(a/X,Xq/a;q,p)_{|\bk|}}
\prod\limits_{i=1}^r\frac{\big(X x_iq^2,Xx_iq^2/bc;q^2,p\big)_{k_i}}{\big(Xx_iq^2/b,Xx_iq^2/c;q^2,p\big)_{k_i}}.
\end{gather*}
This holds when $a^2bc=X^2q$. In the inverse identity \eqref{gba}, we make the substitutions
$x_i\mapsto tx_i$, $1\leq i\leq r$, $a\mapsto a/t$, $b\mapsto aqb/d$, $c\mapsto aq^2/bd$,
where $t$ is chosen so that after these substitutions $t^{r+1}=q	a^2/cX$. This gives the following quadratic summation formula.

\begin{Theorem}\label{aqst}
If $a^2q^{2|\bn|+1}=cd$, then
\begin{gather*}
\sum_{\substack{0\leq k_i\leq n_i,\\ i=1,\dots,r}}
\frac{\Delta\big(\bx q^{2\bk};p\big)}{\Delta(\bx;p)}
\frac{(b,q/b;q,p)_{|\bk|}}{(aq/c,aq/d;q,p)_{|\bk|}} q^{|\bk|}\prod\limits_{i,j=1}^r
\frac{\big(q^{-2n_j}x_i/x_j;q^2,p\big)_{k_i}}{\big(q^2x_i/x_j;q^2,p\big)_{k_i}}
\\ \phantom{\sum_{\substack{0\leq k_i\leq n_i,\\ i=1,\dots,r}}=}{}
\times\prod\limits_{i=1}^r\bigg(\frac{\theta\big(ax_iq^{|\bk|+2k_i};p\big)}{\theta(ax_i;p)}
\frac{(ax_i;q,p)_{|\bk|}}{\big(ax_i q^{2n_i+1};q,p\big)_{|\bk|}}\frac{\big(cx_i,dx_i;q^2,p\big)_{k_i}}
{\big(ax_iq^2/b,abx_iq;q^2,p\big)_{k_i}}\bigg)
\\ \phantom{\sum_{\substack{0\leq k_i\leq n_i,\\ i=1,\dots,r}}}{}
=\frac{\big(aq^2/bc,abq/c;q^2,p\big)_{|\bn|}}{(aq/c;q,p)_{2|\bn|}}\prod\limits_{i=1}^r\frac{(ax_iq;q,p)_{2n_i}} {\big(ax_iq^2/b,abx_iq;q^2,p\big)_{n_i}}.
\end{gather*}
\end{Theorem}

 Theorem~\ref{aqst} has the following companion identity.

\begin{Corollary}\label{aqsc}
If $a^2q=b_1\dotsm b_{r+2}x_1\dotsm x_r$, then
\begin{gather*}
\sum_{\substack{k_1,\dots,k_r\geq 0,\\ |\bk|\leq N}}
\frac{\Delta\big(\bx q^{2\bk};p\big)}{\Delta(\bx;p)}
\frac{\big(q^{-N},q^{N+1};q,p\big)_{|\bk|}}{\prod\limits_{i=1}^{r+2}(aq/b_i;q,p)_{|\bk|}} q^{|\bk|}
\\ \phantom{\sum_{\substack{k_1,\dots,k_r\geq 0,\\ |\bk|\leq N}}=}{}
\times\prod\limits_{i=1}^r
\raisebox{-1pt}{$\left(\rule{0pt}{26pt}\right.$}
\frac{\theta\big(ax_iq^{|\bk|+2k_i};p\big)}{\theta(ax_i;p)}\frac{(ax_i;q,p)_{|\bk|}
\prod\limits_{j=1}^{r+2}\big(x_ib_j;q^2,p\big)_{k_i}}{\big(ax_iq^{N+2},ax_iq^{1-N};q^2,p\big)_{k_i}
\prod\limits_{j=1}^r\big(q^2x_i/x_j;q^2,p\big)_{k_i}}
\raisebox{-1pt}{$\left.\rule{0pt}{26pt}\right)$}
\\ \phantom{\sum_{\substack{k_1,\dots,k_r\geq 0,\\ |\bk|\leq N}}}{}
=\begin{cases}\displaystyle\prod\limits_{i=1}^r\frac{\big(ax_iq^2;q^2,p\big)_m}{\big(q/ax_i;q^2,p\big)_m}
\prod\limits_{i=1}^{r+2}\frac{\big(b_iq/a;q^2,p\big)_m}{\big(aq^2/b_i;q^2,p\big)_m}, & N=2m,
\\[2.5ex]
\displaystyle\prod\limits_{i=1}^r\frac{\big(ax_iq;q^2,p\big)_m}{\big(1/ax_i;q^2,p\big)_m}
\prod\limits_{i=1}^{r+2}\frac{\big(b_i/a;q^2,p\big)_m}{\big(aq/b_i;q^2,p\big)_m}, & N=2m-1.
\end{cases}
\end{gather*}
\end{Corollary}

The $p=0$ cases of Theorem~\ref{aqst} and Corollary~\ref{aqsc} are due to the second author \cite[Theorems~5.4 and~5.8]{s}. In the case $r=1$, they are equivalent to Theorem~\ref{aqist} and hence again reduce to summations from \cite{w}.

\subsection[Dr summations]{$\boldsymbol{D_r}$ summations}\label{dss}

A $D_r$ quadratic summation formula is obtained if we replace
$q$ by $q^2$ in \eqref{cr87gl} and then make the substitutions
$(a,b,c,d,e)\mapsto \big(1,q/ab,b/a,aq^{|\bn|},aq^{|\bn|+1}\big)$.
The resulting identity takes the form~\eqref{fab}, where $F_{\bn\bk}$ is as in the case $m=2$ of Corollary~\ref{bcmic} and
\begin{gather*}
a_{\bk}=\prod\limits_{i=1}^r\frac{(ax_iq;q,p)_{2k_i}\big(x_iq/ab,bx_i/a;q^2,p\big)_{k_i}}
{(x_i/a;q,p)_{2k_i}\big(abx_iq,ax_iq^2/b;q^2,p\big)_{k_i}},
\\
b_{\bk}=\frac{\big(bq^{1-|\bk|},q^{2-|\bk|}/b,a^2q;q^2,p\big)_{|\bk|}
\prod\limits_{i, j=1}^r\big(x_ix_jq^2;q^2,p\big)_{k_i}}{\prod\limits_{i=1}^r
\big(abx_iq,ax_iq^2/b,x_iq^{2-|\bk|}/a,aq^{1+|\bk|-2k_i}/x_i;q^2,p\big)_{k_i}
\prod\limits_{1\le i<j\le r}\big(x_ix_jq^2;q^2,p\big)_{k_i+k_j}}.
\end{gather*}
The inverse identity \eqref{gba} is easily simplified as follows.

\begin{Theorem}\label{dqst}
We have the quadratic summation formula
\begin{gather*}
\sum_{\substack{0\leq k_i\leq n_i,\\ i=1,\dots,r}}
\frac{\Delta\big(\bx q^{2\bk};p\big)}{\Delta(\bx;p)}
q^{|\bk|-e_2(\bk)}
\prod\limits_{1\le i<j\le r}\frac 1{\big(x_ix_j;q^2,p\big) _{k_i+k_j}}
\\ \phantom{\sum_{\substack{0\leq k_i\leq n_i,\\ i=1,\dots,r}}=}
\times\prod\limits_{i=1}^{r}\bigg(\frac{\ta\big(ax_iq^{2k_i+|\bk|};p\big)}{\ta(ax_i;p)}
\frac {(ax_i;q,p)_{|\bk|} (aq/x_i;q,p)_{|\bk|-k_i}}
{\big(ax_iq^{2n_i+1},aq^{1-2n_i}/x_i;q,p\big)_{|\bk|}\big(q^{k_i-|\bk|}x_i/a;q,p\big)_{k_i}}\bigg)
\\ \phantom{\sum_{\substack{0\leq k_i\leq n_i,\\ i=1,\dots,r}}=}{}
\times\frac{\big(a^2q;q^2,p\big)_{|\bk|} (b,q/b;q,p)_{|\bk|}}
{\prod\limits_{i=1}^{r}\big(abx_iq,ax_iq^2/b;q^2,p\big)_{k_i}}\;
\prod\limits_{i,j=1} ^{r}\frac{\big(q^{-2n_j}x_i/x_j,x_ix_jq^{2n_j};q^2,p\big)_{k_i}}
{\big(q^2x_i/x_j;q^2,p\big)_{k_i}}
\\ \phantom{\sum_{\substack{0\leq k_i\leq n_i,\\ i=1,\dots,r}}}{}
=\prod\limits_{i=1} ^{r}\frac{(ax_iq;q,p)_{2n_i} \big(x_iq/ab,bx_i/a;q^2,p\big)_{n_i}}
{(x_i/a;q,p)_{2n_i}\big(abx_iq,ax_iq^2/b;q^2,p\big)_{n_i}}.
\end{gather*}
\end{Theorem}

Theorem~\ref{dqst} has two companion identities where the summation is on a simplex. In the first one, we assume that $b=q^{-N}$ and use the standard argument
to replace the parameters $q^{-2n_j}/x_j$ by generic parameters $b_j$.

\begin{Corollary}\label{dqsc1}
We have the quadratic summation formula
\begin{gather*}
\sum_{\substack{k_1,\dots,k_r\ge 0,\\ |\bk|\le N}}
\frac{\Delta\big(\bx q^{2\bk};p\big)}{\Delta(\bx;p)}
q^{|\bk|-e_2(\bk)}
\frac{\big(a^2q;q^2,p\big)_{|\bk|} \big(q^{N+1},q^{-N};q,p\big)_{|\bk|}}
{\prod\limits_{i=1}^{r}\big(ax_iq^{N+2},ax_iq^{1-N};q^2,p\big)_{k_i}}
\\ \phantom{\sum_{\substack{k_1,\dots,k_r\ge 0,\\ |\bk|\le N}}=}{}
\times\prod\limits_{1\le i<j\le r}\frac 1{\big(x_ix_j;q^2,p\big) _{k_i+k_j}}
\prod\limits_{i,j=1}^{r}\frac{\big(b_jx_i,x_i/b_j;q^2,p\big)_{k_i}}
{\big(q^2x_i/x_j;q^2,p\big)_{k_i}}
\\ \phantom{\sum_{\substack{k_1,\dots,k_r\ge 0,\\ |\bk|\le N}}=}{}
\times\prod\limits_{i=1}^{r}\bigg(\frac{\ta\big(ax_iq^{2k_i+|\bk|};p\big)}{\ta(ax_i;p)}
\frac {(ax_i;q,p)_{|\bk|}
(aq/x_i;q,p)_{|\bk|-k_i}}
{(aq/b_i,ab_iq;q,p)_{|\bk|}\big(q^{k_i-|\bk|}x_i/a;q,p\big)_{k_i}}\bigg)
\\ \phantom{\sum_{\substack{k_1,\dots,k_r\ge 0,\\ |\bk|\le N}}}{}
=\begin{cases}
\displaystyle\prod\limits_{i=1}^{r}\frac
{\big(ax_iq^2,aq^2/x_i,b_iq/a,q/ab_i;q^2,p\big)_m}
{\big(q/ax_i,x_iq/a,aq^2/b_i,ab_iq^2;q^2,p\big)_m},&N=2m,
\\[2.5ex]
\displaystyle\prod\limits_{i=1}^{r}\frac
{\big(ax_iq,aq/x_i,b_i/a,1/ab_i;q^2,p\big)_m}
{\big(1/ax_i,x_i/a,aq/b_i,ab_iq;q^2,p\big)_m},&N=2m-1.
\end{cases}
\end{gather*}
\end{Corollary}

In the second companion identity to Theorem~\ref{dqst}, we first let $a=q^{-N-\frac 12}$ and then replace~$x_i$ by $aq^{N+\frac 12}x_i$ for $1\leq i\leq r$, before applying the standard argument.

\begin{Corollary}\label{dqsc2}
We have the quadratic summation formula
\begin{gather*}
\sum_{\substack{k_1,\dots,k_r\ge 0,\\ |\bk|\le N}}
\frac{\Delta\big(\bx q^{2\bk};p\big)}{\Delta(\bx;p)}
q^{|\bk|-e_2(\bk)}
\prod\limits_{1\le i<j\le r}\frac 1{\big(a^2x_ix_jq^{2N+1};q^2,p\big) _{k_i+k_j}}
\\ \phantom{\sum_{\substack{k_1,\dots,k_r\ge 0,\\ |\bk|\le N}}=}{}
\times
\prod\limits_{i=1}^{r}\bigg(\frac{\ta\big(ax_iq^{2k_i+|\bk|};p\big)}{\ta(ax_i;p)}
\frac {(ax_i;q,p)_{|\bk|}
\big(q^{-2N}/ax_i;q,p\big)_{|\bk|-k_i}}
{\big(aq/c_i,c_iq^{-2N}/a;q,p\big)_{|\bk|}\big(q^{2N+1+k_i-|\bk|}ax_i;q,p\big)_{k_i}}\bigg)
\\ \phantom{\sum_{\substack{k_1,\dots,k_r\ge 0,\\ |\bk|\le N}}=}{}
\times\frac{(b,q/b;q,p)_{|\bk|}\big(q^{-2N};q^2,p\big)_{|\bk|}}
{\prod\limits_{i=1}^{r}\big(abx_iq,ax_iq^2/b;q^2,p\big)_{k_i}}
\prod\limits_{i,j=1}^{r}\frac{\big(c_jx_i,a^2q^{2N+1}x_i/c_j;q^2,p\big)_{k_i}}
{\big(q^2x_i/x_j;q^2,p\big)_{k_i}}
\\ \phantom{\sum_{\substack{k_1,\dots,k_r\ge 0,\\ |\bk|\le N}}}{}
=\prod\limits_{i=1} ^{r}\frac{(ax_iq;q,p)_{2N}\big(abq/c_i,aq^2/bc_i;q^2,p\big)_N}
{(a/c_i;q,p)_{2N}\big(ax_iq^2/b,abx_iq;q^2,p\big)_N}.
\end{gather*}
\end{Corollary}

The cases $p=0$ of Theorem~\ref{dqst}, Corollaries~\ref{dqsc1} and \ref{dqsc2}
are due to the second author \mbox{\cite[Theorems~5.21, 5.25 and 5.27]{s}}.
In the case $r=1$, they are equivalent to Theorem~\ref{aqist} and hence again reduce to summations from \cite{w}.

A $D_r$ cubic summation formula is obtained if we replace
$q$ by $q^3$ in \eqref{cr87gl} and then make the substitutions
$(a,b,c,d,e)\mapsto \big(1,aq^{|\bn|},aq^{|\bn|+1},aq^{|\bn|+2},1/a^3\big)$.
The resulting identity takes the~form~\eqref{fab}, where $F_{\bn\bk}$ is as in the case $m=3$ of Corollary~\ref{bcmic} and
\begin{gather*}
a_{\bk}=
\prod\limits_{i=1}^r\frac{(ax_iq;q,p)_{3k_i}\big(x_i/a^3;q^3,p\big)_{k_i}}
{(x_i/a;q,p)_{3k_i}\big(a^3x_iq^3;q^3,p\big)_{k_i}},
\\[-1ex]
b_{\bk}=\frac{\prod\limits_{i, j=1}^r\big(x_ix_jq^3;q^3,p\big)_{k_i}}{\prod\limits_{1\le i<j\le r}\big(x_ix_jq^3;q^3,p\big)_{k_i+k_j}}
\frac{\big(q^{1-2|\bk|}/a^2,a^2q^{3-|\bk|},a^2q^{2-|\bk|};q^3,p\big)_{|\bk|}}{\prod\limits_{i=1} ^{r}
\big(x_iq^{1-|\bk|}/a,x_iq^{3-|\bk|}/a,a^3x_iq^3,aq^{1+|\bk|-3k_i}/x_i;q^3,p\big)_{k_i}}.
\end{gather*}
The inverse identity \eqref{gba} is easily simplified as follows.

\begin{Theorem}\label{cqst}
We have the cubic summation formula
\begin{gather*}
\sum_{\substack{0\leq k_i\leq n_i,\\i=1,\dots,r}}
\frac{\Delta\big(\bx q^{3\bk};p\big)}{\Delta(\bx;p)}q^{|\bk|-2e_2(\bk)}
\prod\limits_{1\le i<j\le r}\frac 1{(x_ix_j;q^3,p) _{k_i+k_j}}
\\[-1ex] \phantom{\sum_{\substack{0\leq k_i\leq n_i,\\i=1,\dots,r}}}
\times\prod\limits_{i=1}^{r}\bigg(\frac{\ta\big(ax_iq^{|\bk|+3k_i};p\big)}{\ta(ax_i;p)}
\frac {(ax_i;q,p)_{|\bk|}
(aq/x_i;q,p)_{|\bk|-k_i}}
{\big(ax_iq^{3n_i+1},aq^{1-3n_i}/x_i;q,p\big)_{|\bk|}\big(q^{k_i-|\bk|}x_i/a;q,p\big)_{2k_i}}\bigg)
\\[-1ex] \phantom{\sum_{\substack{0\leq k_i\leq n_i,\\i=1,\dots,r}}}
\times\frac{\big(1/a^2;q,p\big)_{|\bk|}\big(a^2q;q,p\big)_{2|\bk|}}
{\prod\limits_{i=1}^{r}\big(a^3x_iq^3;q^3,p\big)_{k_i}}\;
\prod\limits_{i,j=1} ^{r}\frac{\big(q^{-3n_j}x_i/x_j,x_ix_jq^{3n_j};q^3,p\big)_{k_i}}
{\big(q^3x_i/x_j;q^3,p\big)_{k_i}}
\\[-1ex] \phantom{\sum_{\substack{0\leq k_i\leq n_i,\\i=1,\dots,r}}}
=\prod\limits_{i=1} ^{r}\frac{(ax_iq;q,p)_{3n_i}\big(x_i/a^3;q^3,p\big)_{n_i}}
{(x_i/a;q,p)_{3n_i}\big(a^3x_iq^3;q^3,p\big)_{n_i}}.
\end{gather*}
\end{Theorem}

We will give two companion identities to Theorem~\ref{cqst}. For the first one, we first let $a=q^{N/2}$ and then make the substitutions
$x_i\mapsto a x_iq^{-N/2}$ for $1\leq i\leq r$. The standard argument gives the~following identity.

\begin{Corollary}\label{cssa}
We have the cubic summation formula
\begin{gather*}
\sum_{\substack{k_1,\dots,k_r\geq 0\\|\bk |\leq N}}
\frac{\Delta\big(\bx q^{3\bk};p\big)}{\Delta(\bx;p)}q^{|\bk|-2e_2(\bk)}
\prod\limits_{1\le i<j\le r}\frac 1{\big(a^2x_ix_jq^{-N};q^3,p\big) _{k_i+k_j}}
\\[-1ex] \phantom{\sum_{\substack{k_1,\dots,k_r\geq 0\\|\bk |\leq N}}=}{}
\times
\prod\limits_{i=1}^{r}\bigg(\frac{\ta\big(ax_iq^{|\bk|+3k_i};p\big)}{\ta(ax_i;p)}
\frac {(ax_i;q,p)_{|\bk|}
\big(q^{N+1}/ax_i;q,p\big)_{|\bk|-k_i}}
{\big(aq/b_i,b_iq^{N+1}/a;q,p\big)_{|\bk|}\big(q^{k_i-|\bk|-N}ax_i;q,p\big)_{2k_i}}\bigg)
\\[-1ex] \phantom{\sum_{\substack{k_1,\dots,k_r\geq 0\\|\bk |\leq N}}=}{}
\times
\frac{\big(q^{-N};q,p\big)_{|\bk|} \big(q^{N+1};q,p\big)_{2|\bk|}}
{\prod\limits_{i=1}^{r}\big(ax_iq^{N+3};q^3,p\big)_{k_i}}
\prod\limits_{i,j=1}^{r}\frac{\big(b_jx_i,a^2x_iq^{-N}/b_j;q^3,p\big)_{k_i}}
{\big(q^3x_i/x_j;q^3,p\big)_{k_i}}
\\[-1ex] \phantom{\sum_{\substack{k_1,\dots,k_r\geq 0\\|\bk |\leq N}}}{}
=\prod\limits_{i=1}^{r}\frac{\big(b_i/a;q,p\big)_{N+1} \big(q^{-N}/ax_i;q^3,p\big)_{N+1}}
{(1/ax_i;q,p)_{N+1} \big(b_iq^{-N}/a;q^3,p\big)_{N+1}}.
\end{gather*}
\end{Corollary}

Next, we let $a=q^{-(N+1)/2}$ in Theorem~\ref{cqst}, so that the factor $\big(a^2q;q,p\big)_{2|\bk|}$ vanishes unless $2|\bk|\leq N$.
As before, this gives the following companion identity.

\begin{Corollary}\label{cssb}
We have the cubic summation formula
\begin{gather*}
\sum_{\substack{k_1,\dots,k_r\geq 0\\ 2|\bk |\leq N}}
\frac{\Delta\big(\bx q^{3\bk};p\big)}{\Delta(\bx;p)}q^{|\bk|-2e_2(\bk)}
\prod\limits_{1\le i<j\le r}\frac 1{\big(a^2x_ix_jq^{N+1};q^3,p\big) _{k_i+k_j}}
\\ \phantom{\sum_{\substack{k_1,\dots,k_r\geq 0\\ 2|\bk |\leq N}}=}{}
\times
\prod\limits_{i=1}^{r}\bigg(\frac{\ta\big(ax_iq^{|\bk|+3k_i};p\big)}{\ta(ax_i;p)}
\frac {(ax_i;q,p)_{|\bk|} \big(q^{-N}/ax_i;q,p\big)_{|\bk|-k_i}}
{\big(aq/b_i,b_iq^{-N}/a;q,p\big)_{|\bk|}\big(q^{k_i-|\bk|+N+1}ax_i;q,p\big)_{2k_i}}\bigg)
\\ \phantom{\sum_{\substack{k_1,\dots,k_r\geq 0\\ 2|\bk |\leq N}}=}{}
\times
\frac{\big(q^{N+1};q,p\big)_{|\bk|} \big(q^{-N};q,p\big)_{2|\bk|}}
{\prod\limits_{i=1}^{r}\big(ax_iq^{2-N};q^3,p\big)_{k_i}}
\prod\limits_{i,j=1} ^{r}\frac{\big(b_jx_i,a^2x_iq^{N+1}/b_j;q^3,p\big)_{k_i}}
{\big(q^3x_i/x_j;q^3,p\big)_{k_i}}
\\ \phantom{\sum_{\substack{k_1,\dots,k_r\geq 0\\ 2|\bk |\leq N}}}{}
=\prod\limits_{i=1}^{r}\frac
{(ax_iq;q,p)_{N} \big(aq^{2-N}/b_i;q^3,p\big)_{N}}
{(aq/b_i;q,p)_{N} \big(ax_iq^{2-N};q^3,p\big)_{N}}.
\end{gather*}
\end{Corollary}

The cases $p=0$ of Theorem~\ref{cqst}, Corollaries~\ref{cssa} and~\ref{cssb}
are due to the second author \mbox{\cite[Theorems~5.29 and~5.34]{s}}.
The case $r=1$ of Theorem~\ref{cqst} is equivalent to the case $b=a$ of~\cite[Corollary~4.5]{w} and the corollaries to \cite[Corollaries~4.14 and~4.12]{w}, respectively.
The case when
both $p=0$ and $r=1$ of all these results
are due to Gasper \cite{g}.

Finally, to obtain a
 quartic summation from \eqref{cr87gl}, we first replace $q$ by $q^4$ and
 then make the substitutions
 $(a,b,c,d,e)\mapsto \big(1,a q^{|\bn|},a q^{|\bn|+1},aq^{|\bn|+2},aq^{|\bn|+3}\big)$.
This is only possible if $a^4=q^{-2}$. The case $a=q^{-1/2}$ leads to a
trivial identity, so we will only consider the case $a=\ti q^{-1/2}$. We~find that \eqref{fab} holds, where $F_{\bn\bk}$ is as in Corollary~\ref{bcmic} with $m=4$ and $a=\ti q^{-1/2}$, and
\begin{gather*}
a_{\bk}=\prod\limits_{i=1}^r
\frac{\big(\ti x_i q^{\frac 12};q,p\big)_{4k_i}}{\big({-}\ti x_iq^{\frac 12};q,p\big)_{4k_i}},
\\
b_{\bk}=\frac{\prod\limits_{i,j=1}^r\big(x_ix_jq^4;q^4,p\big)_{k_i}}{\prod\limits_{1\leq i<j\leq r}\big(x_ix_jq^4;q^4,p\big)_{k_i+k_j}}
\\ \phantom{a_{\bk}=}{}
\times\frac{\big({-}q^{2-2|\bk|},-q^{3-2|\bk|},-q^{4-2|\bk|};q^4,p\big)_{|\bk|}}{\prod\limits_{i=1}^r\big({-}\ti x_iq^{\frac 52-|\bk|},-\ti x_iq^{\frac 72-|\bk|},-\ti x_iq^{\frac 92-|\bk|},\ti q^{|\bk|-4k_i+\frac 52}/x_i;q^4,p\big)_{k_i}}
\\ \phantom{a_{\bk}}{}
=q^{-3|\bk|-3e_2(\bk)}\big({-}1,-q,-q^2;q^2,p\big)_{|\bk|}
\frac{\prod\limits_{i,j=1}^r\big(x_ix_jq^4;q^4,p\big)_{k_i}}{\prod\limits_{1\leq i<j\leq r}\big(x_ix_jq^4;q^4,p\big)_{k_i+k_j}}
\\ \phantom{a_{\bk}=}{}
\times
\prod\limits_{i=1}^r\bigg(\frac{\theta\big(\ti q^{-\frac 12}/x_i;p\big)}{\theta\big(\ti q^{|\bk|-4k_i+\frac 12}/x_i;p\big)}\frac{\big(\ti q^{\frac 12}/x_i;q,p\big)_{|\bk|-k_i}}{\big(\ti q^{-\frac 12}/x_i;q,p\big)_{|\bk|}\big({-}\ti x_i q^{k_i-|\bk|+\frac 12};q,p\big)_{3k_i}}\bigg),
\end{gather*}
where the second expression for $b_{\bk}$ is obtained by some routine manipulations. In the corresponding identity \eqref{gba}, we reintroduce the parameter $a$ by substituting $x_i\mapsto -\ti a x_iq^{\frac 12}$, $1\leq i\leq r$. This leads to the following result.

\begin{Theorem}\label{cqust}
We have the quartic summation formula
\begin{gather*}
\sum_{\substack{0\leq k_i\leq n_i,\\ i=1,\dots,r}}
\frac{\Delta\big(\bx q^{4\bk};p\big)}{\Delta(\bx;p)}
q^{|\bk|-3e_2(\bk)}
\frac {\big({-}1,-q,-q^2;q^2,p\big)_{|\bk|}}{\prod\limits_{1\le i<j\le r}\big({-}a^2x_ix_jq;q^4,p\big) _{k_i+k_j}}
\\ \phantom{\sum_{\substack{0\leq k_i\leq n_i,\\ i=1,\dots,r}}}{}
\times
\prod\limits_{i=1}^{r}\bigg(\frac{\ta\big(ax_iq^{|\bk|+4k_i};p\big)}{\ta(ax_i;p)}
\frac {(ax_i;q,p)_{|\bk|} (-1/ax_i;q,p)_{|\bk|-k_i}}
{\big(ax_iq^{4n_i+1},-q^{-4n_i}/ax_i;q,p\big)_{|\bk|}\big({-}q^{k_i-|\bk|+1}ax_i;q,p\big)_{3k_i}}\bigg)
\\ \phantom{\sum_{\substack{0\leq k_i\leq n_i,\\ i=1,\dots,r}}}{}
\times
\prod\limits_{i,j=1} ^{r}\frac{\big(q^{-4n_j}x_i/x_j,-a^2x_ix_jq^{4n_j+1};q^4,p\big)_{k_i}}
{\big(q^4x_i/x_j;q^4,p\big)_{k_i}}
=\prod\limits_{i=1} ^{r}\frac{(ax_iq;q,p)_{4n_i}}
{(-ax_iq;q,p)_{4n_i}}.
\end{gather*}
\end{Theorem}

Theorem~\ref{cqust} is a new result even when $p=0$. In fact, we are not aware of any other quartic summations for multiple hypergeometric series. The case $r=1$ (where we have put $x_1=1$ without loss of generality) is
\begin{gather}\label{wqs}
\sum_{k=0}^n\frac{\theta\big(aq^{5k};p\big)}{\theta(a;p)}
\frac{(a;q,p)_k(-1;q,p)_{2k}\big(q^{-4n},-a^2q^{4n+1};q^4,p\big)_k}
{\big(aq^{4n+1},-q^{-4n}/a;q,p\big)_k(-aq;q,p)_{3k}\big(q^2;q^2,p\big)_k} q^k
=\frac{(aq;q,p)_{4n}}{(-aq;q,p)_{4n}}.
\end{gather}
This is the case $b=-aq$ of Warnaar's ``less appealing quartic transformation'', that is, the penultimate identity in \cite{w}. More explicitly, it is given as \cite[Corollary~2.5]{cj}. The case $p=0$ of~\eqref{wqs} is equivalent to the special case $ab=-q^{1-2n}$ of~\cite[equation~(5.28)]{g}.

\subsection*{Acknowledgements}
We thank the anonymous referees for several useful comments.
The first author was partially supported by the Swedish Science Research Council.
The second author was partially supported by Austrian
Science Fund grants S9607 and P32305.


\pdfbookmark[1]{References}{ref}
\LastPageEnding

\end{document}